\documentclass[12pt]{article}
\usepackage{amsmath,amssymb,amsthm}
\newtheorem{thm}{Theorem}[section]

 \newtheorem{cor}{Corollary}[section]
 \newtheorem{lem}{Lemma}[section]
 \newtheorem{prop}{Proposition}[section]
 \newtheorem{defn}{Definition}[section]
\newtheorem{rem}{Remark}[section]

\begin{document}
\begin{center}
{\large{\bf Global existence and optimal decay rates for the Timoshenko system: the case of equal wave speeds}}
\end{center}
\begin{center}
\footnotesize{Naofumi Mori}\\[2ex]
\footnotesize{Graduate School of Mathematics,\\ Kyushu University, Fukuoka 819-0395, Japan}\\
\footnotesize{n-mori@math.kyushu-u.ac.jp}\\

\vspace{3mm}

\footnotesize{Jiang Xu}\\[2ex]
\footnotesize{Department of Mathematics, \\ Nanjing
University of Aeronautics and Astronautics, \\
Nanjing 211106, P.R.China,}\\
\footnotesize{jiangxu\underline{ }79@nuaa.edu.cn}\\

\vspace{3mm}

\footnotesize{Faculty of Mathematics, \\ Kyushu University, Fukuoka 819-0395, Japan}\\

\vspace{5mm}

\footnotesize{Shuichi Kawashima}\\[2ex]
\footnotesize{Faculty of Mathematics, \\ Kyushu University, Fukuoka 819-0395, Japan,}\\
\footnotesize{kawashim@math.kyushu-u.ac.jp}\\
\end{center}
\vspace{6mm}

\begin{abstract}
This work first gives the global existence and optimal decay rates of solutions to the classical Timoshenko system on the framework of Besov spaces.
Due to the \textit{non-symmetric} dissipation, the general theory for dissipative hyperbolic systems (\cite{XK1})
can not be applied to the Timoshenko system directly.
In the case of equal wave speeds, we construct global solutions to the Cauchy problem pertaining to data
in the spatially Besov spaces. Furthermore, the dissipative structure
enables us to give a new decay framework which pays less attention on the traditional spectral analysis. Consequently, the optimal decay estimates
of solution and its derivatives of fractional order are shown
by time-weighted energy approaches in terms of low-frequency and high-frequency decompositions. As a by-product,
the usual decay estimate of $L^{1}(\mathbb{R})$-$L^{2}(\mathbb{R})$ type is also shown.
\end{abstract}


\section{Introduction}
Consider the following Timoshenko system (see \cite{T1,T2}), which is a set of two coupled wave equations of the form
\begin{align}
\left\{\begin{array}{l}
       \varphi_{tt}-(\varphi_x-\psi)_x=0,\\[2mm]
       \psi_{tt}-\sigma(\psi_{x})_{x}-(\varphi_x-\psi)+\gamma \psi_t =0,
       \end{array}\right. \label{R-E1}
\end{align}
and describes the transverse vibrations of a beam. Here $t\geq 0$ is the time variable, $x\in \mathbb{R}$ is the spacial
variable which denotes the point on the center line of the beam,
$\varphi(t,x)$ is the transversal displacement of the beam from an equilibrium state, and $\psi$ is the rotation
angle of the filament of the beam. The smooth function $\sigma(\eta)$ satisfies $\sigma'(\eta)>0$ for any $\eta\in\mathbb{R}$, and $\gamma$
is a positive constant.
System \eqref{R-E1} is supplemented
with the initial data
\begin{equation}
(\varphi, \varphi_{t}, \psi, \psi_{t})(x,0)
=(\varphi_{0}, \varphi_{1}, \psi_{0}, \psi_{1})(x).\label{R-E2}
\end{equation}
The linearized system  of \eqref{R-E1} reads correspondingly as
\begin{align}
\left\{\begin{array}{l}
       \varphi_{tt}-(\varphi_x-\psi)_x=0,\\[2mm]
       \psi_{tt}-a^2\psi_{xx}-(\varphi_x-\psi)+\gamma \psi_t =0,
       \end{array}\right. \label{R-E3}
\end{align}
with $a>0$ is the sound speed defined by $a^2=\sigma'(0)$. The case $a=1$ corresponds to the Timoshenko system with equal wave speeds.

\subsection{Known results}
In a bounded domain, the decay property of \eqref{R-E3} was studied by Rivera and Rake \cite{RR1,RR2}, where they proved
that the energy decayed exponentially as $t\rightarrow0$ if $a=1$ and the one decayed polynomially as $t\rightarrow0$ if $a\neq1$.
In whole space, the third author \textit{et. al.} \cite{IHK} introduced
the following quantities
\begin{align}\label{R-E4}
v=\varphi_x-\psi, \quad
u=\varphi_t, \quad
z=a\psi_x, \quad
y=\psi_t,
\end{align}
so that the system \eqref{R-E3} can be rewritten as
\begin{align}\label{R-E5}
\left\{\begin{array}{l}
        v_t-u_x+y=0,\\[2mm]
        u_t-v_x=0,\\[2mm]
        z_t-ay_x=0,\\[2mm]
         y_t-az_x-v+\gamma y=0.
\end{array}\right.
\end{align}
The initial data are given by
\begin{align}\label{R-E6}
(v, u, z, y)(x, 0)
=(v_{0}, u_{0}, z_{0}, y_{0})(x),
\end{align}
where $v_0=\varphi_{0,x}-\psi_{0}$, $y_0=\psi_1$,
$u_0=\varphi_1$ and $z_0=a\psi_{0,x}$.
Furthermore, it was shown by \cite{IHK} that the dissipative structure of \eqref{R-E5}
is characterized by
\begin{equation}
\begin{cases}
{\rm Re}\,\lambda(i\xi)\leq -c\eta_1(\xi)
   \qquad {\rm for} \quad a=1, \\[1mm]
{\rm Re}\,\lambda(i\xi)\leq -c\eta_2(\xi)
   \qquad {\rm for} \quad a\neq 1,
\end{cases}\label{R-E7}
\end{equation}
where $\lambda(i\xi)$ denotes the eigenvalues of the system
\eqref{R-E5} in the Fourier space, $\eta_1(\xi)=\frac{\xi^2}{1+\xi^2}$,
$\eta_2(\xi)=\frac{\xi^2}{(1+\xi^2)^2}$, and $c$ is a positive
constant. As the consequence, the following decay properties are shown
for $U=(v,u,z,y)^{\top}$ of \eqref{R-E5}:
\begin{align}\label{R-E8}
\|\partial_x^k U(t)\|_{L^2}
\leq
C(1+t)^{-\frac{1}{4}-\frac{k}{2}}\|U_0\|_{L^1}
+Ce^{-ct}\|\partial_x^kU_0\|_{L^2}
\end{align}
for $a=1$, and
\begin{align}\label{R-E9}
\|\partial_x^k U(t)\|_{L^2}
\leq
C(1+t)^{-\frac{1}{4}-\frac{k}{2}}\|U_0\|_{L^1}
+C(1+t)^{-\frac{l}{2}}\|\partial_x^{k+l}U_0\|_{L^2}
\end{align}
for $a\neq 1$, where $U_0:=(v_0,z_0,u_0,y_0)$, $k$ and $l$ are nonnegative integers, and $c$ and $C$ are positive constants.
However, the energy functionals in \cite{IHK} are not optimal. Recently, by a careful analysis for asymptotic expansions of the eigenvalues,
the first author and third author \cite{MK} gave the optimal energy method in Fourier spaces, which is regarded as an improved version
of that in \cite{IHK}. Recently, with the additional assumption $\int_{\mathbb{R}}U_{0}dx=0$, Racke and Said-Houari \cite{RS} strengthened those decay properties in \cite{IHK} such that linearized solutions
decay faster with a rate of $t^{-\gamma/2}$, by introducing the integral space $L^{1,\gamma}(\mathbb{R})$.

Other studies on the dissipative Timoshenko system can be found in the literature.
We refer to \cite{RBS,RFSC} for frictional dissipation case, \cite{FR,SAJM,SK} for thermal dissipation case, and \cite{ABMR,ARMSV,LK,LP} for memory-type dissipation case.

\subsection{Main results}
The main aim of this paper is to establish the global existence and optimal decay estimates of
solutions in spatially critical Besov spaces. To the best of our knowledge,
so far there is no results available in this direction for the Timoshenko system, although
the critical space has already been succeeded in the study of fluid dynamical equations, see \cite{AGZ,D1,H,PZ} for Navier-Stokes equations,
\cite{D2,XW,XY} for Euler equations and related models.
In \cite{XK1}, with the assumptions of dissipative entropy and Shizuta-Kawashima condition, the second and third authors
have already studied generally dissipative hyperbolic systems where the dissipation matrix is symmetric, however, the Timoshenko system has the non-symmetric dissipation.
More precisely, with the aid of variable change \eqref{R-E4} (with $a=1$), it is convenient to rewrite \eqref{R-E1}-\eqref{R-E2} as a Cauchy problem for
the hyperbolic system of first order
\begin{align}\label{R-E10}
\left\{\begin{array}{l}
        U_t + A(U)U_x + LU =0,\\[2mm]
        U(x, 0) = U_0(x),
\end{array}\right.
\end{align}
where
\begin{align}\label{R-E11}
A(U)=-\left(
\begin{array}{cccc}
0 & 1 & 0 & 0 \\
1 & 0 & 0 & 0 \\
0 & 0 & 0 & 1 \\
0 & 0 & \sigma'(z) & 0
\end{array}
\right),\ \ \
L=\left(
\begin{array}{cccc}
0 & 0 & 0 & 1 \\
0 & 0 & 0 & 0 \\
0 & 0 & 0 & 0 \\
-1 & 0 & 0 & \gamma
\end{array}
\right).
\end{align}

Notice that $A(U)$ is a real symmetrizable matrix due to $\sigma'(z)>0$, and the matrix $L$ is nonnegative definite but not symmetric,
so the Timoshenko system \eqref{R-E10} is an example of hyperbolic systems
with non-symmetric dissipation. Consequently, those general results (see \cite{XK1}) for hyperbolic systems with symmetric dissipation
can not be applied directly, which is the main motivation of this paper.

The partial damping term $\gamma y$ is a weak dissipation,
which enables us to capture the dissipation from contributions of
$(y,v,u_{x},z_{x})$ only. However, there is no dissipative rates for $u,z$ themselves.
To overcome the difficulty in the derivation of a priori estimates, an elementary
fact well developed (see \cite{XK1}) that indicates the relation between homogeneous and inhomogeneous Chemin-Lerner
spaces, will be used, see proofs of Lemmas \ref{lem3.1}-\ref{lem3.4} for more details.
On the other hand, the second and third authors gave a new decay
framework for general dissipative system satisfying the Shizuta-Kawashima condition (see \cite{XK2}),
which allows to pay less attention on the traditional spectral analysis. Inspired by the dissipative
structure for the Timoshenko system (see \cite{IHK} or \eqref{R-E7}), we hope that the new decay framework can be adapted to
the Timoshenko system with equal wave speeds. However, those analysis remain valid only for the case of high dimension $(n\geq3)$ due to
interpolation techniques. To overcome this obstruction, the degenerate space $\dot{B}^{-1/2}_{2,\infty}$ rather than the general form $\dot{B}^{-s}_{2,\infty}(0<s\leq1/2)$ will be employed. Let us mention that $L^{1}(\mathbb{R})\hookrightarrow\dot{B}^{0}_{1,\infty}(\mathbb{R})\hookrightarrow\dot{B}^{-1/2}_{2,\infty}(\mathbb{R})$.
Additionally, we involve new observations in order to achieve the optimal decay estimates at the low-frequency, see \eqref{R-E78} and \eqref{R-E81} for details.

In the present paper, we focus on the Timoshenko system with equal wave speeds ($a=1$).
Now, main results are stated as follows.
\begin{thm}\label{thm1.1}
Suppose that $U_{0}\in B^{3/2}_{2,1}(\mathbb{R})$. There exists a positive constant $\delta_{0}$ such that if
$$\|U_{0}\|_{B^{3/2}_{2,1}(\mathbb{R})}\leq
\delta_{0}, $$
then the Cauchy problem \eqref{R-E10} has a unique
global classical solution $U\in \mathcal{C}^{1}(\mathbb{R}^{+}\times
\mathbb{R})$ satisfying
$$U \in\widetilde{\mathcal{C}}(B^{3/2}_{2,1}(\mathbb{R}))\cap\widetilde{\mathcal{C}}^{1}(B^{1/2}_{2,1}(\mathbb{R}))$$
Moreover, the following energy inequality holds that
\begin{eqnarray*}
&&\|U\|_{\widetilde{L}^\infty(B^{3/2}_{2,1}(\mathbb{R}))}+\Big(\|y\|_{\widetilde{L}^2_{T}(B^{3/2}_{2,1})}+\|(v,z_{x})\|_{\widetilde{L}^2_{T}(B^{1/2}_{2,1})}
+\|u_{x}\|_{\widetilde{L}^2_{T}(B^{-1/2}_{2,1})}\Big)\nonumber\\&\leq&  C_{0}\|U_{0}\|_{B^{3/2}_{2,1}(\mathbb{R})},
\end{eqnarray*}
where $C_{0}>0$ is a constant.
\end{thm}

\begin{rem}
To the best of our knowledge, Theorem \ref{thm1.1} exhibit the optimal critical regularity of global well-posedness for \eqref{R-E10}, which
is the first result in this direction for the Timoshenko system. Observe that there is 1-regularity-loss phenomenon for the dissipation rates
due to the nonlinear influence, which is totally different in comparison with the linearized system \eqref{R-E5} with $a=1$.
\end{rem}

Based on the global-in-time existence of solutions, we further obtain the optimal decay estimates.
Denote $\Lambda^{\alpha}f:=\mathcal{F}^{-1}|\xi|^{\alpha}\mathcal{F}f (\alpha\in \mathbb{R})$.
\begin{thm}\label{thm1.2}
Let $U(t,x)=(v,u,z,y)(t,x)$ be the global classical solution of Theorem \ref{thm1.1}. If further the initial data $U_{0}\in \dot{B}^{-1/2}_{2,\infty}(\mathbb{R})$ and
$$\mathcal{M}_{0}:=\|U_{0}\|_{B^{3/2}_{2,1}(\mathbb{R})\cap\dot{B}^{-1/2}_{2,\infty}(\mathbb{R})}$$
is sufficiently small. Then the classical solution $U(t,x)$ of \eqref{R-E10} admits the following decay estimates
\begin{eqnarray}
\|\Lambda^{\ell}U\|_{X_{1}(\mathbb{R})}\lesssim \mathcal{M}_{0}(1+t)^{-\frac{1}{4}-\frac{\ell}{2}} \label{R-E12}
\end{eqnarray}
for $0\leq\ell\leq1/2$, where $X_{1}:=B^{1/2-\ell}_{2,1}$ if $0\leq\ell<1/2$ and $X_{1}:=\dot{B}^{0}_{2,1}$ if $\ell=1/2$.
\end{thm}

Note that the $L^1(\mathbb{R})$ embedding property in Lemma \ref{lem2.3}, as an immediate by-product of Theorem \ref{thm1.2}, the usual optimal decay estimates of $L^{1}(\mathbb{R})$-$L^{2}(\mathbb{R})$ type are available.

\begin{cor}\label{cor1.1}
Let $U(t,x)=(v,u,z,y)(t,x)$ be the global classical solutions of Theorem \ref{thm1.1}.
If further the initial data $U_{0}\in L^1(\mathbb{R})$ and
$$\widetilde{\mathcal{M}}_{0}:=\|U_{0}\|_{B^{3/2}_{2,1}(\mathbb{R})\cap L^1(\mathbb{R})}$$
is sufficiently small, then
\begin{eqnarray}
\|\Lambda^{\ell}U\|_{L^2(\mathbb{R})}\lesssim \widetilde{\mathcal{M}}_{0}(1+t)^{-\frac{1}{4}-\frac{\ell}{2}} \label{R-E13}
\end{eqnarray}
for $0\leq\ell\leq 1/2$.
\end{cor}

\begin{rem}
Let us mention that Theorem \ref{thm1.2} and Corollary \ref{cor1.1} exhibit various decay rates of solution and its derivatives of fractional order.
In comparison with \cite{IK}, here, the harmonic analysis allows to reduce significantly the regularity requirements on the initial data.
\end{rem}

The rest of this paper unfolds as follows. In Sect. \ref{sec:2}, we
present useful properties in Besov spaces, which will be used in the subsequence
analysis. In Sect. \ref{sec:3}, we construct the global-in-time solution
by Fourier localization energy methods. Based on the dissipative structure, in Sect. \ref{sec:4}, we
develop the decay property for the linearized Timoshenko system \eqref{R-E4}-\eqref{R-E5}
on the framework of Besov spaces. Then, by employing localized time-weighted energy approaches,
we deduce the optimal decay estimates for \eqref{R-E9}. In Sect. \ref{sec:5} (Appendix), we present those definitions for Besov spaces and
Chemin-Lerner spaces for the convenience of reader.

Finally, we explain some notations. Throughout the paper, $f\lesssim g$ denotes $f\leq Cg$, where $C>0$
is a generic constant. $f\thickapprox g$ means $f\lesssim g$ and $g\lesssim f$. Denote by $\mathcal{C}([0,T],X)$ (resp.,
$\mathcal{C}^{1}([0,T],X)$) the space of continuous (resp.,
continuously differentiable) functions on $[0,T]$ with values in a
Banach space $X$. Also, $\|(f,g,h)\|_{X}$ means $
\|f\|_{X}+\|g\|_{X}+\|h\|_{X}$, where $f,g,h\in X$.

\section{Tools}\setcounter{equation}{0}\label{sec:2}
In this section, we only present analysis properties in Besov spaces and Chemin-Lerner spaces in $\mathbb{R}^{n}(n\geq1)$, which will be used in the sequence section.
For convenience of reader, the Appendix is devoted to those definitions for Besov spaces and Chemin-Lerner spaces.

Firstly, we give an improved
Bernstein inequality (see, \textit{e.g.}, \cite{W}), which allows the case of fractional derivatives.

\begin{lem}\label{lem2.1}
Let $0<R_{1}<R_{2}$ and $1\leq a\leq b\leq\infty$.
\begin{itemize}
\item [(i)] If $\mathrm{Supp}\mathcal{F}f\subset \{\xi\in \mathbb{R}^{n}: |\xi|\leq
R_{1}\lambda\}$, then
\begin{eqnarray*}
\|\Lambda^{\alpha}f\|_{L^{b}}
\lesssim \lambda^{\alpha+n(\frac{1}{a}-\frac{1}{b})}\|f\|_{L^{a}}, \ \  \mbox{for any}\ \  \alpha\geq0;
\end{eqnarray*}

\item [(ii)]If $\mathrm{Supp}\mathcal{F}f\subset \{\xi\in \mathbb{R}^{n}:
R_{1}\lambda\leq|\xi|\leq R_{2}\lambda\}$, then
\begin{eqnarray*}
\|\Lambda^{\alpha}f\|_{L^{a}}\approx\lambda^{\alpha}\|f\|_{L^{a}}, \ \  \mbox{for any}\ \ \alpha\in\mathbb{R}.
\end{eqnarray*}
\end{itemize}
\end{lem}

Besov spaces obey various inclusion relations. Precisely,
\begin{lem}\label{lem2.2} Let $s\in \mathbb{R}$ and $1\leq
p,r\leq\infty,$ then
\begin{itemize}
\item[(1)]If $s>0$, then $B^{s}_{p,r}=L^{p}\cap B^{s}_{p,r};$
\item[(2)]If $\tilde{s}\leq s$, then $B^{s}_{p,r}\hookrightarrow
B^{\tilde{s}}_{p,r}$. This inclusion relation is false for
the homogeneous Besov spaces;
\item[(3)]If $1\leq r\leq\tilde{r}\leq\infty$, then $\dot{B}^{s}_{p,r}\hookrightarrow
\dot{B}^{s}_{p,\tilde{r}}$ and $B^{s}_{p,r}\hookrightarrow
B^{s}_{p,\tilde{r}};$
\item[(4)]If $1\leq p\leq\tilde{p}\leq\infty$, then $\dot{B}^{s}_{p,r}\hookrightarrow \dot{B}^{s-n(\frac{1}{p}-\frac{1}{\tilde{p}})}_{\tilde{p},r}
$ and $B^{s}_{p,r}\hookrightarrow
B^{s-n(\frac{1}{p}-\frac{1}{\tilde{p}})}_{\tilde{p},r}$;
\item[(5)]$\dot{B}^{n/p}_{p,1}\hookrightarrow\mathcal{C}_{0},\ \ B^{n/p}_{p,1}\hookrightarrow\mathcal{C}_{0}(1\leq p<\infty);$
\end{itemize}
where $\mathcal{C}_{0}$ is the space of continuous bounded functions
which decay at infinity.
\end{lem}

\begin{lem}\label{lem2.3}
Suppose that $\varrho>0$ and $1\leq p<2$. It holds that
\begin{eqnarray*}
\|f\|_{\dot{B}^{-\varrho}_{r,\infty}}\lesssim \|f\|_{L^{p}}
\end{eqnarray*}
with $1/p-1/r=\varrho/n$. In particular, this holds with $\varrho=n/2, r=2$ and $p=1$.
\end{lem}

The global existence depends on a key fact, which indicates the connection between
homogeneous Chemin-Lerner spaces and inhomogeneous Chemin-Lerner
spaces, see \cite{XK1} for the proof.  Precisely,
\begin{prop} \label{prop2.1}
Let $s\in \mathbb{R}$ and $1\leq \theta, p,r\leq\infty$.
\begin{itemize}
\item [$(1)$] It holds that
\begin{eqnarray*}
L^{\theta}_{T}(L^{p})\cap
\widetilde{L}^{\theta}_{T}(\dot{B}^{s}_{p,r})\subset \widetilde{L}^{\theta}_{T}(B^{s}_{p,r});
\end{eqnarray*}
\item [$(2)$] Furthermore, as $s>0$ and $\theta\geq r$, it holds that
\begin{eqnarray*}
L^{\theta}_{T}(L^{p})\cap
\widetilde{L}^{\theta}_{T}(\dot{B}^{s}_{p,r})=\widetilde{L}^{\theta}_{T}(B^{s}_{p,r})
\end{eqnarray*}
\end{itemize}
for any $T>0$.
\end{prop}

Let us state the Moser-type product estimates, which plays an important role in the estimate of bilinear
terms.
\begin{prop}\label{prop2.2}
Let $s>0$ and $1\leq
p,r\leq\infty$. Then $\dot{B}^{s}_{p,r}\cap L^{\infty}$ is an algebra and
$$
\|fg\|_{\dot{B}^{s}_{p,r}}\lesssim \|f\|_{L^{\infty}}\|g\|_{\dot{B}^{s}_{p,r}}+\|g\|_{L^{\infty}}\|f\|_{\dot{B}^{s}_{p,r}}.
$$
Let $s_{1},s_{2}\leq n/p$ such that $s_{1}+s_{2}>n\max\{0,\frac{2}{p}-1\}. $  Then one has
$$\|fg\|_{\dot{B}^{s_{1}+s_{2}-n/p}_{p,1}}\lesssim \|f\|_{\dot{B}^{s_{1}}_{p,1}}\|g\|_{\dot{B}^{s_{2}}_{p,1}}.$$
\end{prop}

In the sequel we also need a estimate for commutator.
\begin{prop}\label{prop2.3}
Let  $1<p<\infty, 1\leq \theta \leq\infty$ and $\
s\in(-\frac{n}{p}-1, \frac{n}{p}]$. Then there exists a generic
constant $C>0$ depending only on $s, N$ such that
$$\begin{cases}
\|[f,\dot{\Delta}_{q}]g\|_{L^{p}}\leq
Cc_{q}2^{-q(s+1)}\|f\|_{\dot{B}^{\frac{n}{p}+1}_{p,1}}\|g\|_{\dot{B}^{s}_{p,1}},\cr
\|[f,\dot{\Delta}_{q}]g\|_{L^{\theta}_{T}(L^{p})}\leq
Cc_{q}2^{-q(s+1)}\|f\|_{\widetilde{L}^{\theta_{1}}_{T}(\dot{B}^{\frac{n}{p}+1}_{p,1})}\|g\|_{\widetilde{L}^{\theta_{2}}_{T}(\dot{B}^{s}_{p,1})},\end{cases}
$$
with $1/\theta=1/\theta_{1}+1/\theta_{2}$, where the commutator
$[\cdot,\cdot]$ is defined by $[f,g]=fg-gf$ and $\{c_{q}\}$ denotes
a sequence such that $\|(c_{q})\|_{ {l^{1}}}\leq 1$.
\end{prop}

Finally, we state a continuity result for compositions (see,\textit{ e.g.}, \cite{H}) to end this section.
\begin{prop}\label{prop2.4}
Let $s>0$, $1\leq p, r, \theta\leq \infty$, $F\in
W^{[s]+3,\infty}_{loc}(I;\mathbb{R})$ with $F(0)=0$, $T\in
(0,\infty]$ and $f\in \widetilde{L}^{\theta}_{T}(B^{s}_{p,r})\cap
L^{\infty}_{T}(L^{\infty}).$ Then there exists a function $C$ depending only on $s,p,r,n,$ and $F$ such that
$$\begin{cases}
\|F(f)-F'(0)f\|_{\dot{B}^{s}_{p,r}}\leq
C(\|f\|_{L^{\infty}})\|f\|^2_{\dot{B}^{s}_{p,r}},\cr
\|F(f)-F'(0)f\|_{\widetilde{L}^{\theta}_{T}(\dot{B}^{s}_{p,r})}\leq
C(\|f\|_{L^{\infty}_{T}(L^{\infty})})\|f\|^2_{\widetilde{L}^{\theta}_{T}(\dot{B}^{s}_{p,r})}.
\end{cases}
$$
\end{prop}

\section{Global-in-time existence}\setcounter{equation}{0}\label{sec:3}
Recently, the second and third authors \cite{XK1} have already established a local existence theory for generally symmetric hyperbolic systems in
spatially critical Besov spaces, which is viewed as the generalization of the basic theory of Kato and Majda
\cite{K,M}. The new result can be applied to the current problem (\ref{R-E10}), since the non-symmetric dissipation does not work for the local-in-time existence.
Precisely,
\begin{prop}\label{prop3.1} Assume that
$U_{0}\in{B^{3/2}_{2,1}}$, then there exists a time
$T_{0}>0$ (depending only on the initial data) such that
\begin{itemize}
\item[(i)] (Existence):
system (\ref{R-E10}) has a unique solution
$U(t,x)\in\mathcal{C}^{1}([0,T_{0}]\times \mathbb{R})$ satisfying
$U\in\widetilde{\mathcal{C}}_{T_{0}}(B^{3/2}_{2,1})\cap
\widetilde{\mathcal{C}}^1_{T_{0}}(B^{1/2}_{2,1})$;
\item[(ii)] (Blow-up criterion): if the maximal time $T^{*}(>T_{0})$ of existence of such a solution
is finite, then
$$\limsup_{t\rightarrow T^{*}}\|U(t,\cdot)\|_{B^{3/2}_{2,1}}=\infty$$
if and only if $$\int^{T^{*}}_{0}\|\nabla
U(t,\cdot)\|_{L^{\infty}}dt=\infty.$$
\end{itemize}
\end{prop}

Furthermore, in order to show that classical solutions in Proposition \ref{prop3.1} are globally defined, the next task is to
construct a priori estimates according to the dissipative mechanism of the Tomoshenko system.
To this end,
we define by $E(T)$ the energy functional and by $D(T)$ the corresponding dissipation functional:
$$E(T):=\|U\|_{\widetilde{L}^\infty_{T}(B^{3/2}_{2,1})}$$
and
$$D(T):=\|y\|_{\widetilde{L}^2_{T}(B^{3/2}_{2,1})}+\|(v,z_{x})\|_{\widetilde{L}^2_{T}(B^{1/2}_{2,1})}
+\|u_{x}\|_{\widetilde{L}^2_{T}(B^{-1/2}_{2,1})}$$
for any time $T>0$.

The first lemma is related to nonlinear a priori estimates for the dissipation for $y$.

\begin{lem}(The dissipation for $y$)\label{lem3.1}
If $U\in\widetilde{\mathcal{C}}_{T}(B^{3/2}_{2,1})\cap
\widetilde{\mathcal{C}}^1_{T}(B^{1/2}_{2,1})$ is a solution of
(\ref{R-E10}) for any $T>0$, then
\begin{eqnarray}
&&E(T)+\|y\|_{\widetilde{L}^{2}_{T}(B^{3/2}_{2,1})}\lesssim \|U_{0}\|_{B^{3/2}_{2,1}}+\sqrt{E(T)}D(T). \label{R-E14}
\end{eqnarray}
\end{lem}
\begin{proof}
Multiplying the equations in (\ref{R-E10}) by $v$, $u$, $\sigma(z)-\sigma(0)$ and $y$, respectively, and then adding the resulting equalities to get
\begin{eqnarray}
\frac{1}{2}\frac{d}{dt}(v^2+y^2+u^2+S(z))-\Big(vu+[\sigma(z)-\sigma(0)]\Big)_{x}+\gamma y^2=0, \label{R-E15}
\end{eqnarray}
where $$S(z)=2\int^{z}_{0}\Big(\sigma(\eta)-\sigma(0)\Big)d\eta.$$
Note that $S(z)$ is equivalent to $z^2$,  due to the fact $\sigma'(\eta)>0$ and the smallness assumption (\ref{R-E52}) below. Then we perform the integral to (\ref{R-E15})
with respect to $x$ and obtain the basic energy equality
\begin{eqnarray}
\frac{1}{2}\frac{d}{dt}E_{0}(U) +\gamma\|y\|_{L^2}^2=0, \label{R-E16}
\end{eqnarray}
where the energy functional $E_{0}(U)$ is defined by
$$E_{0}(U)=\|(v,u,y)\|^2_{L^2}+\int_{\mathbb{R}}S(z)dx \approx\|U\|^2_{L^2}.$$

Integrating in $t\in [0,T]$ and taking the square-root of the resulting inequality, we arrive at
\begin{eqnarray}
\|U\|_{L^\infty_{T}(L^2)}+\sqrt{2\gamma}\|y\|_{L^2_{T}(L^2)}\leq\|U_{0}\|_{L^2} \label{R-E17}
\end{eqnarray}
for any $T>0$.

Next, we perform the frequency-localization estimate and get the dissipation rate from $y$ in homogeneous Chemin-Lerner spaces.
Applying the operator $\dot{\Delta}_{q}(q\in\mathbb{Z})$ to (\ref{R-E10}), we have
\begin{eqnarray}\label{R-E18}
\left\{\begin{array}{l}
        \dot{\Delta}_{q}v_t-\dot{\Delta}_{q}u_x+\dot{\Delta}_{q}y=0,\\[1mm]
        \dot{\Delta}_{q}u_t-\dot{\Delta}_{q}v_x=0,\\[1mm]
        \dot{\Delta}_{q}z_t-\dot{\Delta}_{q}y_x=0,\\ [1mm]
        \dot{\Delta}_{q}y_t-\sigma'(z)\dot{\Delta}_{q}z_x-\dot{\Delta}_{q}v+\gamma \dot{\Delta}_{q}y=[\dot{\Delta}_{q},\sigma'(z)]z_x,
\end{array}\right.
\end{eqnarray}
where the commutator is defined by $[f,g]:=fg-gf$. Multiplying (\ref{R-E18}) with $\dot{\Delta}_{q}v$, $\dot{\Delta}_{q}u,$ $ \sigma'(z)\dot{\Delta}_{q}z$ and $\dot{\Delta}_{q}y$, respectively, and then adding the resulting equalities to get
\begin{eqnarray}
&&\frac{1}{2}\frac{d}{dt}\Big(|\dot{\Delta}_{q}v|^2+|\dot{\Delta}_{q}y|^2+|\dot{\Delta}_{q}u|^2+\sigma'(z)|\dot{\Delta}_{q}z|^2\Big)\\ \nonumber
&&-\Big\{(\dot{\Delta}_{q}u\dot{\Delta}_{q}v)_{x}+\Big(\sigma'(z)\dot{\Delta}_{q}z\dot{\Delta}_{q}y\Big)_{x}\Big\}+\gamma|\dot{\Delta}_{q}y|^2
\\ \nonumber
&=&\frac{1}{2}\sigma'(z)_{t}|\dot{\Delta}_{q}z|^2-\sigma'(z)_{x}\dot{\Delta}_{q}z\dot{\Delta}_{q}y+[\dot{\Delta}_{q},\sigma'(z)]
z_{x}\dot{\Delta}_{q}y. \label{R-E19}
\end{eqnarray}
Furthermore, by employing the integral with respect to $x$, with the aid of Cauchy-Schwarz inequality, we obtain
\begin{eqnarray}
&&\frac{1}{2}\frac{d}{dt}E_{0}[\dot{\Delta}_{q}U] +\gamma\|\dot{\Delta}_{q}y\|_{L^2}^2\\ \nonumber
&\lesssim& \|\sigma'(z)_{t}\|_{L^\infty}\|\dot{\Delta}_{q}z\|^2_{L^2}+\|\sigma'(z)_{x}\|_{L^\infty}\|\dot{\Delta}_{q}z\|_{L^2}\|\dot{\Delta}_{q}y\|_{L^2}\\ \nonumber
&&+
\|[\dot{\Delta}_{q},\sigma'(z)]z_{x}\|_{L^2}\|\dot{\Delta}_{q}y\|_{L^2}, \label{R-E20}
\end{eqnarray}
where $$E_{0}[\dot{\Delta}_{q}U]:=\|(\dot{\Delta}_{q}v,\dot{\Delta}_{q}y,\dot{\Delta}_{q}u)\|^2_{L^2}
+\int_{\mathbb{R}}\sigma'(z)|\dot{\Delta}_{q}z|\approx\|\dot{\Delta}_{q}U\|^2_{L^2}.$$
From (\ref{R-E10}) and (\ref{R-E52}), we have
\begin{eqnarray}
\|\sigma'(z)_{t}\|_{L^\infty}\|\dot{\Delta}_{q}z\|^2_{L^2}\lesssim \|z_{t}\|_{L^\infty}\|\dot{\Delta}_{q}z\|^2_{L^2}\lesssim\|y_{x}\|_{L^\infty}\|\dot{\Delta}_{q}z\|^2_{L^2} \label{R-E21}
\end{eqnarray}
Similarly,
\begin{eqnarray}
\|\sigma'(z)_{x}\|_{L^\infty}\|\dot{\Delta}_{q}z\|_{L^2}\|\dot{\Delta}_{q}y\|_{L^2}
\lesssim\|z_{x}\|_{L^\infty}\|\dot{\Delta}_{q}z\|_{L^2}\|\dot{\Delta}_{q}y\|_{L^2}. \label{R-E22}
\end{eqnarray}

Together with (\ref{R-E21})-(\ref{R-E22}), by integrating in $t\in [0,T]$, with the help of Young's inequality, we are led to
\begin{eqnarray}
&&\sqrt{E_{0}[\dot{\Delta}_{q}U]}+\sqrt{2\gamma}\|\dot{\Delta}_{q}y\|_{L^2_{T}(L^2)}
\nonumber\\
&\lesssim& \sqrt{E_{0}[\dot{\Delta}_{q}U_{0}]}+\sqrt{\|(y_{x},z_{x})\|_{L^\infty_{T}(L^\infty)}}\Big(\|\dot{\Delta}_{q}y\|_{L^2_{T}(L^2)}+\|\dot{\Delta}_{q}z\|_{L^2_{T}(L^2)}\Big)
\nonumber\\
&&+\sqrt{\|[\dot{\Delta}_{q},\sigma'(z)]z_{x}\|_{L^2_{T}(L^2)}\|\dot{\Delta}_{q}y\|_{L^2_{T}(L^2)}}. \label{R-E23}
\end{eqnarray}
It follows from the commutator estimate in Proposition \ref{prop2.3} that
\begin{eqnarray}
\|[\dot{\Delta}_{q},\sigma'(z)]
z_{x}\|_{L^2_{T}(L^2)}\lesssim c_{q}2^{-\frac{3 q}{2}}\|z\|_{\widetilde{L}^{\infty}_{T}(\dot{B}^{3/2}_{2,1})}\|z_{x}\|_{\widetilde{L}^{2}_{T}(\dot{B}^{1/2}_{2,1})}, \label{R-E24}
\end{eqnarray}
where $\{c_{q}\}$ denotes a sequence such that $\|c_{q}\|_{\ell^{1}}\leq1$. Therefore, we obtain

\begin{eqnarray}
&&2^{\frac{3q}{2}}\|\dot{\Delta}_{q}U\|_{L^\infty_{T}(L^2)}+\sqrt{2\gamma}2^{\frac{3q}{2}}\|\dot{\Delta}_{q}y\|_{L^2_{T}(L^2)}\nonumber\\& \lesssim &
\|\dot{\Delta}_{q}U_{0}\|_{L^2}+c_{q}\sqrt{\|(y_{x},z_{x})\|_{L^\infty_{T}(\dot{B}^{1/2}_{2,1})}}
\Big(\|y\|_{\widetilde{L}^{2}_{T}(\dot{B}^{3/2}_{2,1})}+\|z_{x}\|_{\widetilde{L}^{2}_{T}(\dot{B}^{1/2}_{2,1})}\Big)
\nonumber\\&&+c_{q}\sqrt{\|z\|_{\widetilde{L}^\infty_{T}(\dot{B}^{3/2}_{2,1})}}
\Big(\|y\|_{\widetilde{L}^{2}_{T}(\dot{B}^{3/2}_{2,1})}+\|z_{x}\|_{\widetilde{L}^{2}_{T}(\dot{B}^{1/2}_{2,1})}\Big), \label{R-E25}
\end{eqnarray}
Here and below, each $\{c_{q}\}$ satisfies $\|c_{q}\|_{\ell^{1}}\leq1$ although it is possibly different in (\ref{R-E25}). Hence, summing up  on $q\in \mathbb{Z}$, we arrive at
\begin{eqnarray}
&&\|U\|_{\widetilde{L}^\infty_{T}(\dot{B}^{3/2}_{2,1})}+\sqrt{2\gamma}\|y\|_{\widetilde{L}^{2}_{T}(\dot{B}^{3/2}_{2,1})}\nonumber\\& \lesssim &
\|U_{0}\|_{\dot{B}^{3/2}_{2,1}}+\sqrt{\|(y,z)\|_{\widetilde{L}^\infty_{T}(\dot{B}^{3/2}_{2,1})}}\Big(\|y\|_{\widetilde{L}^{2}_{T}(\dot{B}^{3/2}_{2,1})}
+\|z_{x}\|_{\widetilde{L}^{2}_{T}(\dot{B}^{1/2}_{2,1})}\Big). \label{R-E26}
\end{eqnarray}

Finally, combining (\ref{R-E17}) and (\ref{R-E26}), it follows from the fact that indicates the connection between
homogeneous Besov spaces and inhomogeneous Besov spaces (Proposition \ref{prop2.1}), we conclude that
\begin{eqnarray}
&&E(T)+\|y\|_{\widetilde{L}^{2}_{T}(B^{3/2}_{2,1})}\lesssim \|U_{0}\|_{B^{3/2}_{2,1}}+\sqrt{E(T)}D(T). \label{R-E27}
\end{eqnarray}
Therefore, the proof of Lemma \ref{lem3.1} is complete.
\end{proof}

\begin{lem}(The dissipation for $v$)\label{lem3.2} If $U\in\widetilde{\mathcal{C}}_{T}(B^{3/2}_{2,1})\cap
\widetilde{\mathcal{C}}^1_{T}(B^{1/2}_{2,1})$ is a solution of
(\ref{R-E10}) for any $T>0$, then
\begin{eqnarray}
\|v\|_{\widetilde{L}^{2}_{T}(B^{1/2}_{2,1})}\lesssim E(T)+\|U_{0}\|_{B^{3/2}_{2,1}}
+\|y\|_{\widetilde{L}^{2}_{T}(B^{3/2}_{2,1})}
+\sqrt{E(T)}D(T). \label{R-E28}
\end{eqnarray}
\end{lem}
\begin{proof}
To do this, it is convenient to rewrite the system (\ref{R-E10}) as follows:
\begin{align}\label{R-E29}
\left\{\begin{array}{l}
        v_t-u_x+y=0,\\[2mm]
        u_t-v_x=0,\\[2mm]
        z_t-y_x=0,\\[2mm]
         y_t-z_{x}-v+\gamma y=g(z)_{x},
\end{array}\right.
\end{align}
where the smooth function $g(z)$ is defined by
$$g(z)=\sigma(z)-\sigma(0)-z=O(z^2)$$ satisfying $g(0)=0$ and $g'(0)=0$.

By multiplying the four equations by $-y, -z,-u$ and $-v$, respectively,  we can deduce that
\begin{eqnarray}
&&\frac{d}{dt}E_{1}(U)+\|v\|^2_{L^2}\leq \|y\|^2_{L^2}+\gamma\|y\|_{L^2}\|v\|_{L^2}+\|g(z)_{x}\|_{L^2}\|v\|_{L^2}, \label{R-E30}
\end{eqnarray}
where
$$E_{1}(U):=-\int_{\mathbb{R}}(vy+uz)dx.$$
It follows from Young's inequality that
\begin{eqnarray}
\frac{d}{dt}E_{1}(U)+\frac{1}{2}\|v\|^2_{L^2}
\lesssim \|y\|^2_{L^2}+\|z\|_{L^\infty}\|z_{x}\|_{L^2}\|v\|_{L^2}. \label{R-E31}
\end{eqnarray}
Integrating (\ref{R-E31}) in $t\in[0,T]$ gives
\begin{eqnarray}
&&\|v\|^2_{L^2_{t}(L^2)}\nonumber\\&\lesssim & (|E_{1}(U)|+|E_{1}(U_{0})|)+\|y\|^2_{L^2_{t}(L^2)}
+\|z\|_{L^\infty_{t}(L^\infty)}\|z_{x}\|_{L^2_{t}(L^2)}\|v\|_{L^2_{t}(L^2)}
\nonumber\\&\lesssim& E(T)^2+\|U_{0}\|^2_{B^{3/2}_{2,1}}+\|y\|^2_{L^2_{T}(L^2)}
+E(T)D^2(T), \label{R-E32}
\end{eqnarray}
for any $T>0$, where we have used the embedding property in Lemma \ref{lem2.2}.

Then, by Young's inequality again, we get
\begin{eqnarray}
\|v\|_{L^2_{T}(L^2)}\lesssim E(T)+\|U_{0}\|_{B^{3/2}_{2,1}}+\|y\|_{L^2_{T}(L^2)}+\sqrt{E(T)}D(T). \label{R-E33}
\end{eqnarray}

Next, we turn to the localization energy estimate.
Applying the operator $\dot{\Delta}_{q}(q\in \mathbb{Z})$ to (\ref{R-E29}) implies that
\begin{align}\label{R-E34}
\left\{\begin{array}{l}
        \dot{\Delta}_{q}v_t-\dot{\Delta}_{q}u_x+\dot{\Delta}_{q}y=0,\\[2mm]
        \dot{\Delta}_{q}u_t-\dot{\Delta}_{q}v_x=0,\\[2mm]
        \dot{\Delta}_{q}z_t-\dot{\Delta}_{q}y_x=0,\\ [2mm]
        \dot{\Delta}_{q}y_t-\dot{\Delta}_{q}z_{x}-\dot{\Delta}_{q}v+\gamma \dot{\Delta}_{q}y=\dot{\Delta}_{q}g(z)_{x}.
\end{array}\right.
\end{align}
Multiplying the first equation in (\ref{R-E34}) by $-\dot{\Delta}_{q}y$, the second one by $-\dot{\Delta}_{q}z$,
the third one by $-\dot{\Delta}_{q}u$ and the fourth one by $-\dot{\Delta}_{q}v$, respectively, then adding the resulting equalities
\begin{eqnarray}
&&-(\dot{\Delta}_{q}v\dot{\Delta}_{q}y+\dot{\Delta}_{q}u\dot{\Delta}_{q}z)_{t} +(\dot{\Delta}_{q}v\dot{\Delta}_{q}z+\dot{\Delta}_{q}u\dot{\Delta}_{q}y)_{x}+|\dot{\Delta}_{q}v|^2\nonumber\\ &=&|\dot{\Delta}_{q}y|^2+
\gamma\dot{\Delta}_{q}y\dot{\Delta}_{q}v-\dot{\Delta}_{q}g(z)_{x}\dot{\Delta}_{q}v. \label{R-E35}
\end{eqnarray}
With the aid of H\"{o}lder and Young's inequalities, we obtain
\begin{eqnarray}
\frac{d}{dt}E_{1}[\dot{\Delta}_{q}U]+\frac{1}{2}\|\dot{\Delta}_{q}v\|^2_{L^2}\lesssim \|\dot{\Delta}_{q}y\|^2_{L^2}+\|\dot{\Delta}_{q}g(z)_{x}\|_{L^2}\|\dot{\Delta}_{q}v\|_{L^2}, \label{R-E36}
\end{eqnarray}
where
$$E_{1}[\dot{\Delta}_{q}U]:=-\int_{\mathbb{R}}(\dot{\Delta}_{q}v\dot{\Delta}_{q}y+\dot{\Delta}_{q}u\dot{\Delta}_{q}z)dx.$$
By performing the integral with respect to $t\in [0,T]$, we are led to
\begin{eqnarray}
&&\|\dot{\Delta}_{q}v\|^2_{L^2_{t}(L^2)}
\nonumber\\  &\lesssim &
\|\dot{\Delta}_{q}U\|^2_{L^\infty_{T}(L^2)}+\|\dot{\Delta}_{q}U_{0}\|^2_{L^2}+\|\dot{\Delta}_{q}y\|^2_{L^2_{T}(L^2)}\nonumber\\ &&
+\|\dot{\Delta}_{q}v\|_{L^\infty_{T}(L^2)}\|\dot{\Delta}_{q}g(z)_{x}\|_{L^1_{T}(L^2)}. \label{R-E366}
\end{eqnarray}
Furthermore, Young's inequality enables us to get
\begin{eqnarray}
&&2^{\frac{q}{2}}\|\dot{\Delta}_{q}v\|_{L^2_{T}(L^2)}\nonumber\\ &\lesssim& c_{q}\|U\|_{\widetilde{L}^{\infty}_{T}(\dot{B}^{1/2}_{2,1})}+c_{q}\|U_{0}\|_{\dot{B}^{1/2}_{2,1}}
\nonumber\\ &&
+c_{q}\|y\|_{\widetilde{L}^{2}_{T}(\dot{B}^{1/2}_{2,1})}
+c_{q}\sqrt{\|v\|_{\widetilde{L}^{\infty}_{T}(\dot{B}^{1/2}_{2,1})}}\|g(z)_{x}\|^{\frac{1}{2}}_{\widetilde{L}^{1}_{T}(\dot{B}^{1/2}_{2,1})}, \label{R-E367}
\end{eqnarray}
where the norm of $g(z)$  on the right-side of (\ref{R-E367})
can be estimated by Proposition \ref{prop2.4} and Remark \ref{rem5.1}
\begin{eqnarray}
\|g(z)_{x}\|_{\widetilde{L}^{1}_{T}(\dot{B}^{1/2}_{2,1})}&\lesssim& \int_{0}^{T}\|g(z)\|_{\dot{B}^{3/2}_{2,1}}dt
\nonumber\\ &\lesssim&
 \int_{0}^{T}\|z\|^2_{\dot{B}^{3/2}_{2,1}}dt \lesssim \|z_{x}\|^2_{\widetilde{L}^{2}_{T}(\dot{B}^{1/2}_{2,1})}. \label{R-E368}
\end{eqnarray}
Therefore, together with (\ref{R-E367})-(\ref{R-E368}), by summing up  on $q\in \mathbb{Z}$,  we arrive at
\begin{eqnarray}
&&\|v\|_{\widetilde{L}^{2}_{T}(\dot{B}^{1/2}_{2,1})}\nonumber\\&\lesssim& \|U\|_{\widetilde{L}^{\infty}_{T}(\dot{B}^{1/2}_{2,1})}+\|U_{0}\|_{\dot{B}^{1/2}_{2,1}}
+\|y\|_{\widetilde{L}^{2}_{T}(\dot{B}^{1/2}_{2,1})}
\nonumber\\&&+\sqrt{\|v\|_{\widetilde{L}^{\infty}_{T}(\dot{B}^{1/2}_{2,1})}}\|z_{x}\|_{\widetilde{L}^{2}_{T}(\dot{B}^{1/2}_{2,1})}.
 \label{R-E37}
\end{eqnarray}
Finally, noticing (\ref{R-E33}) and (\ref{R-E37}), it follows from Proposition \ref{prop2.1} that
\begin{eqnarray}
\|v\|_{\widetilde{L}^{2}_{T}(B^{1/2}_{2,1})}\lesssim E(T)+\|U_{0}\|_{B^{3/2}_{2,1}}
+\|y\|_{\widetilde{L}^{2}_{T}(B^{3/2}_{2,1})}
+\sqrt{E(T)}D(T), \label{R-E38}
\end{eqnarray}
which is just (\ref{R-E28}).
\end{proof}

\begin{lem}\label{lem3.3}
(The dissipation for $z_{x}$) If $U\in\widetilde{\mathcal{C}}_{T}(B^{3/2}_{2,1})\cap
\widetilde{\mathcal{C}}^1_{T}(B^{1/2}_{2,1})$ is a solution of
(\ref{R-E10}) for any $T>0$, then
\begin{eqnarray}
\|z_{x}\|_{\widetilde{L}^{2}_{T}(B^{1/2}_{2,1})} &\lesssim & E(T)+\|U_{0}\|_{B^{3/2}_{2,1}}+\|y\|_{\widetilde{L}^{2}_{T}(B^{3/2}_{2,1})}\nonumber\\ &&+
\|v\|_{\widetilde{L}^{2}_{T}(B^{1/2}_{2,1})}+\sqrt{E(T)}D(T). \label{R-E39}
\end{eqnarray}
\end{lem}
\begin{proof}
Multiplying the third equation in (\ref{R-E29}) by $y_{x}$ and the fourth one by $-z_{x}$ and integrating the resulting equalities over $\mathbb{R}$, we arrive at
\begin{eqnarray}
&&\frac{d}{dt}E_{2}(U)+\|z_{x}\|^2_{L^2}
\nonumber\\ &\lesssim &\|y_{x}\|^2_{L^2}+(\|v\|_{L^2}+\|y\|_{L^2})\|z_{x}\|_{L^2}+\|z\|_{L^\infty}\|z_{x}\|^2_{L^2}, \label{R-E40}
\end{eqnarray}
where $$E_{2}(U):=-\int_{\mathbb{R}}z_{x}ydx.$$
Similar to the procedure leading to (\ref{R-E33}), we arrive at
\begin{eqnarray}
\|z_{x}\|_{L^2_{T}(L^2)}&\lesssim& E(T) +\|U_{0}\|_{B^{3/2}_{2,1}}+\|y\|_{\widetilde{L}^{2}_{T}(B^{3/2}_{2,1})}\nonumber\\ &&+\|v\|_{\widetilde{L}^{2}_{T}(B^{1/2}_{2,1})}+
\sqrt{E(T)}D(T). \label{R-E41}
\end{eqnarray}
On the other hand, from (\ref{R-E34}), we have
\begin{align}\label{R-E42}
\left\{\begin{array}{l}
        \dot{\Delta}_{q}z_t-\dot{\Delta}_{q}y_x=0,\\[2mm]
         \dot{\Delta}_{q}y_t-\dot{\Delta}_{q}z_{x}-\dot{\Delta}_{q}v+\gamma \dot{\Delta}_{q}y=\dot{\Delta}_{q}g(z)_{x}.
\end{array}\right.
\end{align}
Then, by multiplying the first equation in (\ref{R-E42}) by $\dot{\Delta}_{q}y_{x}$ and the second one by $-\dot{\Delta}_{q}z_{x}$, respectively, and then employing the energy estimates on each block, we are led to
\begin{eqnarray}
&&2^{\frac{q}{2}}\|\dot{\Delta}_{q}z_{x}\|_{L^2_{T}(L^2)}\nonumber\\ &\lesssim &
c_{q}(\|U\|_{\widetilde{L}^{\infty}_{T}(B^{3/2}_{2,1})}+\|U_{0}\|_{B^{3/2}_{2,1})})
+c_{q}\|y_{x}\|_{\widetilde{L}^{2}_{T}(\dot{B}^{1/2}_{2,1})}\nonumber\\ &&
+c_{q}\epsilon\|z_{x}\|_{\widetilde{L}^{2}_{T}(\dot{B}^{1/2}_{2,1})}+c_{q}C_{\epsilon}(\|v\|_{\widetilde{L}^{2}_{T}(\dot{B}^{1/2}_{2,1})}+\|y\|_{\widetilde{L}^{2}_{T}(\dot{B}^{1/2}_{2,1})})
\nonumber\\ &&+c_{q}\sqrt{\|z_{x}\|_{\widetilde{L}^{\infty}_{T}(\dot{B}^{1/2}_{2,1})}}\|g(z)_{x}\|^{\frac{1}{2}}_{\widetilde{L}^{1}_{T}(\dot{B}^{1/2}_{2,1})}.
\label{R-E43}
\end{eqnarray}
Furthermore, similar to the estimates (\ref{R-E368})-(\ref{R-E37}), we get
\begin{eqnarray}
&&\|z_{x}\|_{\widetilde{L}^{2}_{T}(\dot{B}^{1/2}_{2,1})}\nonumber\\ &\lesssim & \|U\|_{\widetilde{L}^{\infty}_{T}(B^{3/2}_{2,1})}+\|U_{0}\|_{B^{3/2}_{2,1}}+
\|y\|_{\widetilde{L}^{2}_{T}(\dot{B}^{3/2}_{2,1})}\nonumber\\&&+\|v\|_{\widetilde{L}^{2}_{T}(\dot{B}^{1/2}_{2,1})}+\|y\|_{\widetilde{L}^{2}_{T}(\dot{B}^{1/2}_{2,1})}
+\sqrt{\|z\|_{\widetilde{L}^{\infty}_{T}(\dot{B}^{3/2}_{2,1})}}\|z_{x}\|_{\widetilde{L}^{2}_{T}(\dot{B}^{1/2}_{2,1})}, \label{R-E45}
\end{eqnarray}
where we have chosen $0<\epsilon\leq1/2$.

Finally, by combining (\ref{R-E41}) and (\ref{R-E45}), we arrive at (\ref{R-E39}).
\end{proof}

\begin{lem}\label{lem3.4}
(The dissipation for $u_{x}$) If $U\in\widetilde{\mathcal{C}}_{T}(B^{3/2}_{2,1})\cap
\widetilde{\mathcal{C}}^1_{T}(B^{1/2}_{2,1})$ is a solution of
(\ref{R-E10}) for any $T>0$, then
\begin{eqnarray}
\|u_{x}\|_{\widetilde{L}^{2}_{T}(B^{-1/2}_{2,1})}\lesssim E(T)+\|U_{0}\|_{B^{3/2}_{2,1}}+\|v\|_{\widetilde{L}^{2}_{T}(B^{1/2}_{2,1})}+\|y\|_{\widetilde{L}^{2}_{T}(B^{3/2}_{2,1})}. \label{R-E47}
\end{eqnarray}
\end{lem}

\begin{proof}
Applying the inhomogeneous operator $\Delta_{q}(q\geq-1)$ to the first equation and second one of (\ref{R-E29}) gives
\begin{equation}
\left\{
\begin{array}{l}
\Delta_{q}v_{t}-\Delta_{q}u_{x}+\Delta_{q}y=0,\\
\Delta_{q}u_{t}-\Delta_{q}v_{x}=0.
\end{array} \right.\label{R-E48}
\end{equation}
Multiplying the first equation in (\ref{R-E48}) by $-\Delta_{q}u_{x}$ and the second one by $\Delta_{q}v_{x}$, we can obtain
\begin{eqnarray}
\frac{d}{dt}E_{3}[\Delta_{q}U]+\|\Delta_{q}u_{x}\|^2_{L^2}\leq \|\Delta_{q}v_{x}\|^2_{L^2}+\|\Delta_{q}u_{x}\|_{L^2}\|\Delta_{q}y\|_{L^2}, \label{R-E49}
\end{eqnarray}
where $$E_{3}[\Delta_{q}U]:=-\int_{\mathbb{R}}\Delta_{q}v\Delta_{q}u_{x}dx.$$
Then we integrate (\ref{R-E49}) with respect to $t\in [0,T]$ to get
\begin{eqnarray}
\|\Delta_{q}u_{x}\|^2_{L^2_{t}(L^2)}&\leq& \Big(|E_{3}[\Delta_{q}U]|+E_{3}[\Delta_{q}U_{0}]\Big)\nonumber\\ &&+\|\Delta_{q}v_{x}\|^2_{L^2_{t}(L^2)}
+\|\Delta_{q}u_{x}\|_{L^2_{t}(L^2)}\|\Delta_{q}y\|_{L^2_{t}(L^2)}. \label{R-E50}
\end{eqnarray}
By using Young's inequality and embedding properties in Lemma \ref{lem2.2}, we are led to
\begin{eqnarray}
&&2^{-q/2}\|\Delta_{q}u_{x}\|_{L^2_{T}(L^2)}\nonumber\\ &\lesssim& c_{q}E(T)
+c_{q}\|U_{0}\|_{B^{3/2}_{2,1}}
+c_{q}\|v\|_{\widetilde{L}^{2}_{T}(B^{1/2}_{2,1})}
\nonumber\\&&+c_{q} \sqrt{\|u_{x}\|_{\widetilde{L}^{2}_{T}(B^{-1/2}_{2,1})}\|y\|_{\widetilde{L}^{2}_{T}(B^{3/2}_{2,1})}},\label{R-E51}
\end{eqnarray}
which leads to (\ref{R-E47}) immediately.
\end{proof}

Having Lemmas \ref{lem3.1}-\ref{lem3.4}, we obtain the following a priori estimates for solutions. For brevity, we feel free to skip
the details.

\begin{prop}\label{prop3.2} Suppose $U\in\widetilde{\mathcal{C}}_{T}(B^{3/2}_{2,1})\cap
\widetilde{\mathcal{C}}^1_{T}(B^{1/2}_{2,1})$ is a solution of
(\ref{R-E10}) for $T>0$. There exists $\delta_{1}>0$ such that if
\begin{eqnarray}
E(T)\leq\delta_{1}, \label{R-E52}
\end{eqnarray}
then the following
estimate holds:
\begin{eqnarray}
&&E(T)+D(T)\lesssim \|U_{0}\|_{B^{3/2}_{2,1}}+\sqrt{E(T)}D(T). \label{R-E53}
\end{eqnarray}
Furthermore, it holds that
\begin{eqnarray}
&&E(T)+D(T)\lesssim \|U_{0}\|_{B^{3/2}_{2,1}}. \label{R-E54}
\end{eqnarray}
\end{prop}

By using the standard boot-strap argument, for instance, see
\cite{MN} (Theorem 7.1, p.100), Theorem \ref{thm1.1} follows from the
local existence result (Proposition \ref{prop3.1}) and a
priori estimate (Proposition \ref{prop3.2}). Here, we give the
outline for completeness.\\

\noindent\textbf{\textit{The proof of Theorem \ref{thm1.1}}.}
If the
initial data satisfy $
\|U_{0}\|_{B^{3/2}_{2,1}}\leq\frac{\delta_{1}}{2}$,
by Proposition \ref{prop3.1}, then we determine a time
$T_{1}>0(T_{1}\leq T_{0})$ such that the local solutions of
(\ref{R-E10}) exists in
$\widetilde{\mathcal{C}}_{T_{1}}(B^{3/2}_{2,1})$ and
$\|U\|_{\widetilde{L}^\infty_{T_{1}}(B^{3/2}_{2,1})}\leq\delta_{1}$.
Therefore from Proposition \ref{prop3.2} the solutions satisfy the
a priori estimate
$\|U\|_{\widetilde{L}^\infty_{T_{1}}(B^{3/2}_{2,1})}\leq
C_{1}\|U_{0}\|_{B^{3/2}_{2,1}}\leq\frac{\delta_{1}}{2}$
provided $
\|U_{0}\|_{B^{\sigma}_{2,1}}\leq\frac{\delta_{1}}{2C_{1}}.$
Thus by Proposition \ref{prop3.1} the system
(\ref{R-E10}) for $t\geq T_{1}$ with the initial data
$U(T_{1})$ has again a unique solution
$U$ satisfying
$\|U\|_{\widetilde{L}^\infty_{(T_{1},2T_{1})}(B^{3/2}_{2,1})}\leq\delta_{1}$,
further
$\|U\|_{\widetilde{L}^\infty_{2T_{1}}(B^{3/2}_{2,1})}\leq\delta_{1}$.
Then by Proposition \ref{prop3.2} we have
$\|U\|_{\widetilde{L}^\infty_{2T_{1}}(B^{3/2}_{2,1})}\leq
C_{1}\|U_{0}\|_{B^{3/2}_{2,1}}\leq\frac{\delta_{1}}{2}$.
Subsequently, we continuous the same process for $0\leq t\leq nT_{1},
n=3,4,...$ and finally get a global solution
$U\in\widetilde{\mathcal{C}}(B^{\sigma}_{2,1})$
satisfying
\begin{eqnarray}&&\|U\|_{\widetilde{L}^\infty(B^{3/2}_{2,1})}+\Big(\|y\|_{\widetilde{L}^2_{T}(B^{3/2}_{2,1})}+\|(v,z_{x})\|_{\widetilde{L}^2_{T}(B^{1/2}_{2,1})}
+\|u_{x}\|_{\widetilde{L}^2_{T}(B^{-1/2}_{2,1})})\nonumber\\&\leq&  C_{1}\|U_{0}\|_{B^{3/2}_{2,1}}
\leq\frac{\delta_{1}}{2}.\label{R-E55}
\end{eqnarray}

\section{Optimal decay rates}\setcounter{equation}{0}\label{sec:4}
By employing the energy method in Fourier spaces in \cite{IHK,MK}, it is well-known that the linearized system (\ref{R-E5})-(\ref{R-E6}) admits the dissipative structure
$${\rm Re}\,\lambda(i\xi)\leq -c\eta_1(\xi),\ \ \  {\rm for} \quad a=1, $$
with $\eta_1(\xi)=\frac{\xi^2}{1+\xi^2}$, that is, the following differential inequality holds
\begin{eqnarray}
\frac{d}{dt}E[\hat{U}]+c_{1}\eta_{1}(\xi)|\hat{U}|^2\leq0, \label{R-E56}
\end{eqnarray}
where $E[\hat{U}]\approx|\hat{U}|^2$. As a matter of fact, following from the derivation of (\ref{R-E56}) in \cite{IHK,MK}, we can deduce
the frequency-localization differential inequality
\begin{eqnarray}
\frac{d}{dt}E[\widehat{\Delta_{q}U}]+c_{1}\eta_{1}(\xi)|\widehat{\Delta_{q}U}|^2\leq0, \label{R-E57}
\end{eqnarray}
for $q\geq-1$, and
\begin{eqnarray}
\frac{d}{dt}E[\widehat{\dot{\Delta}_{q}U}]+c_{1}\eta_{1}(\xi)|\widehat{\dot{\Delta}_{q}U}|^2\leq0, \label{R-E58}
\end{eqnarray}
for $q\in \mathbb{Z}$.

\subsection{Decay property for the linearized system}
As shown by \cite{XK2}, we do the similar high-frequency and low-frequency analysis to achieve the following decay property for (\ref{R-E5})-(\ref{R-E6}).

\begin{prop}\label{prop4.1}
If $U_{0}\in \dot{B}^{\sigma}_{2,1}(\mathbb{R})\cap \dot{B}^{-s}_{2,\infty}(\mathbb{R})$ for $\sigma\geq0$ and $s>0$, then the solutions $U(t,x)$ of (\ref{R-E5})-(\ref{R-E6}) has the decay estimate
\begin{equation}
\|\Lambda^{\ell}U\|_{B_{2,1}^{\sigma-\ell}}\lesssim \|U_{0}\|_{\dot{B}_{2,1}^{\sigma}\cap \dot{B}_{2,\infty}^{-s}}(1+t)^{-\frac{\ell+s}{2}} \label{R-E59}
\end{equation}
for $0\leq\ell\leq\sigma$. In particular, if $U_{0}\in \dot{B}^{\sigma}_{2,1}(\mathbb{R})\cap L^p(\mathbb{R})(1\leq p<2$), one further has
\begin{equation}
\|\Lambda^{\ell}U\|_{B_{2,1}^{\sigma-\ell}}\lesssim \|U_{0}\|_{\dot{B}_{2,1}^{\sigma}\cap L^{p}}(1+t)^{-\frac{1}{2}(\frac{1}{p}-\frac{1}{2})-\frac{\ell}{2}} \label{R-E60}
\end{equation}
for $0\leq\ell\leq\sigma$.
\end{prop}

Additionally, we also obtain the decay property on the framework of homogeneous Besov spaces, see \cite{XK2} for the similar proof.
\begin{prop}\label{prop4.2}
If $U_{0}\in \dot{B}^{\sigma}_{2,1}(\mathbb{R})\cap \dot{B}^{-s}_{2,\infty}(\mathbb{R})$ for $\sigma\in \mathbb{R}, s\in \mathbb{R}$ satisfying $\sigma+s>0$, then the solution $U(t,x)$ of (\ref{R-E5})-(\ref{R-E6}) has the decay estimate
\begin{equation}
\|U\|_{\dot{B}_{2,1}^{\sigma}}\lesssim \|U_{0}\|_{\dot{B}_{2,1}^{\sigma}\cap \dot{B}_{2,\infty}^{-s}}(1+t)^{-\frac{\sigma+s}{2}}. \label{R-E61}
\end{equation}
In particular, if $U_{0}\in \dot{B}^{\sigma}_{2,1}(\mathbb{R})\cap L^p(\mathbb{R})(1\leq p<2$), one further has
\begin{equation}
\|U\|_{\dot{B}_{2,1}^{\sigma}}\lesssim \|U_{0}\|_{\dot{B}_{2,1}^{\sigma}\cap L^{p}}(1+t)^{-\frac{1}{2}(\frac{1}{p}-\frac{1}{2})-\frac{\sigma}{2}}. \label{R-E62}
\end{equation}
\end{prop}

\subsection{Localized time-weighted energy approaches}
Firstly, we denote by $\mathcal{G}(t)$ the Green matrix associated with the linearized system (\ref{R-E5})-(\ref{R-E6}) as follows:
\begin{eqnarray}\mathcal{G}(t)f=\mathcal{F}^{-1}[e^{
-t\hat{\Phi}(i\xi)}\mathcal{F}f],  \label{R-E63}
\end{eqnarray}
with
\begin{eqnarray*}
\hat{\Phi}(i\xi)=(i\xi A+L),
\end{eqnarray*} where
\begin{eqnarray*}
A=-\left(%
\begin{array}{ccccc}
 0 & 1 & 0& 0\\
1 & 0 &0 &0\\
0 & 0 &  0 &1\\
0&0&1& 0\\
\end{array}%
\right), \ \ \ L=\left(
\begin{array}{cccc}
0 & 0 & 0 & 1 \\
0 & 0 & 0 & 0 \\
0 & 0 & 0 & 0 \\
-1 & 0 & 0 & \gamma
\end{array}
\right).
\end{eqnarray*}
Then, by the classical Duhamel principle, the solution to Cauchy problem of
the nonlinear Timoshenko system
\begin{eqnarray}\label{R-E64}
\left\{\begin{array}{l}
        v_t-u_x+y=0,\\[2mm]
              u_t-v_x=0,\\[2mm]
        z_t-y_x=0,\\[2mm]
          y_t-z_{x}-v+\gamma y=g(z)_{x},
\end{array}\right.
\end{eqnarray}
with
\begin{eqnarray}
U|_{t=0}=U_{0}(x) \label{R-E65}
\end{eqnarray}
can be represented by
\begin{eqnarray}
U(t,x)=\mathcal{G}(t)U_{0}+\int^{t}_{0}\mathcal{G}(t-\tau)\mathcal{R}(\tau)d\tau, \label{R-E66}
\end{eqnarray}
where $\mathcal{R}:=(0,0,0,g(z)_{x})^{\top}$. Note that the smooth function $g(z)=O(z^2)$ satisfying $g(0)=0$ and $g'(0)=0$.

Additionally, from the definition of $\mathcal{G}(t)$, it is not difficult to obtain the frequency-localization Duhamel principle for
(\ref{R-E64})-(\ref{R-E65}).

\begin{lem}\label{lem4.1}
Suppose that $U(t,x)$ is a solution of (\ref{R-E64})-(\ref{R-E65}). Then
\begin{eqnarray}
\Delta_{q}\Lambda^{\ell}U(t,x)=\Delta_{q}\Lambda^{\ell}[\mathcal{G}(t)U_{0}]
+\int^{t}_{0}\Delta_{q}\Lambda^{\ell}[\mathcal{G}(t-\tau)\mathcal{R}(\tau)]d\tau\label{R-E67}
\end{eqnarray}
for $q\geq-1$ and $\ell\in \mathbb{R}$, and
\begin{eqnarray}
\dot{\Delta}_{q}\Lambda^{\ell}U(t)=\dot{\Delta}_{q}\Lambda^{\ell}[\mathcal{G}(t)U_{0}]
+\int^{t}_{0}\dot{\Delta}_{q}\Lambda^{\ell}[\mathcal{G}(t-\tau)\mathcal{R}(\tau)]d\tau\label{R-E68}
\end{eqnarray}
for $q\in\mathbb{Z}$ and $\ell\in \mathbb{R}$.
\end{lem}

Based on Lemma \ref{lem4.1}, we shall deduce the optimal decay estimate by developing
time-weighted energy approaches as in \cite{Ma} in terms of high-frequency and low-frequency decompositions.
For this purpose, we first define some sup-norms as follows
$$
\mathcal{E}_{0}(t):=\sup_{0\leq\tau\leq t}\|U(\tau)\|_{B^{3/2}_{2,1}};
$$
\begin{eqnarray*}
\mathcal{E}_{1}(t):=\sup_{0\leq\ell<1/2}\sup_{0\leq\tau\leq t}(1+\tau)^{\frac{1}{4}+\frac{\ell}{2}}\|\Lambda^{\ell}U(\tau)\|_{B^{1/2-\ell}_{2,1}}+\sup_{0\leq\tau\leq t}(1+\tau)^{\frac{1}{2}}\|\Lambda^{\frac{1}{2}}U(\tau)\|_{\dot{B}^{0}_{2,1}}.
\end{eqnarray*}
 As a consequence, we
have
\begin{prop}\label{prop4.1}
Let $U=(v,u,z,y)^{\top}$ be the global classical solution in the sense of Theorem \ref{thm1.1}. Suppose that $U_{0}\in B^{3/2}_{2,1}\cap \dot{B}^{-1/2}_{2,\infty}$ and the norm
$\mathcal{M}_{0}:=\|U_{0}\|_{B^{3/2}_{2,1}\cap \dot{B}^{-1/2}_{2,\infty}}$ is sufficiently small. Then it holds that
\begin{eqnarray}
\|\Lambda^{\ell}U(t)\|_{X_{1}}\lesssim \mathcal{M}_{0}(1+t)^{-\frac{1}{4}-\frac{\ell}{2}} \label{R-E69}
\end{eqnarray}
for $0\leq\ell\leq1/2$, where $X_{1}:=B^{1/2-\ell}_{2,1}$ if $0\leq\ell<1/2$ and $X_{1}:=\dot{B}^{0}_{2,1}$ if $\ell=1/2$.
\end{prop}

Actually, Proposition \ref{prop4.1} depends on an energy inequality related to sup-norms $\mathcal{E}_{0}(t)$ and $\mathcal{E}_{1}(t)$, which
is included in the following proposition.
\begin{prop}\label{prop4.2}
Let $U=(v,u,z,y)^{\top}$ be the global classical solution in the sense of Theorem \ref{thm1.1}. Additional, if $U_{0}\in \dot{B}^{-1/2}_{2,\infty}$, then
\begin{eqnarray}
\mathcal{E}_{1}(t)\lesssim \mathcal{M}_{0}+\mathcal{E}_{1}^{2}(t)+\mathcal{E}_{0}(t)\mathcal{E}_{1}(t), \label{R-E70}
\end{eqnarray}
where $\mathcal{M}_{0}$ is the same notation defined in Proposition \ref{prop4.1}.
\end{prop}
\begin{proof}
The proof consists of two steps.\\
\underline{\textit{Step 1: High-frequency estimate}}

Due to $\Delta_{q}f\equiv\dot{\Delta}_{q}f(q\geq0$), it suffices to show the inhomogeneous case. It follows from the
high-frequency analysis for (\ref{R-E5})-(\ref{R-E6}) (see, \textit{e.g.}, \cite{XK2}) that
\begin{eqnarray}
\|\Delta_{q}\Lambda^{\ell}\mathcal{G}(t)U_{0}\|_{L^2}\lesssim e^{-c_{2}t}\|\Delta_{q}\Lambda^{\ell}U_{0}\|_{L^2}\ \  (c_{2}>0) \label{R-E71}
\end{eqnarray}
for all $q\geq0$. Then by Lemma \ref{lem4.1}, we arrive at
\begin{eqnarray}
&&\|\Delta_{q}\Lambda^{\ell}U\|_{L^2}
\nonumber\\&\leq&\|\Delta_{q}\Lambda^{\ell}[\mathcal{G}(t)U_{0}]\|_{L^2}+\int^{t}_{0}\|\Delta_{q}\Lambda^{\ell}[\mathcal{G}(t-\tau)\mathcal{R}(\tau)]\|_{L^2}d\tau
\nonumber\\&\lesssim& e^{-c_{2}t}\|\Delta_{q}\Lambda^{\ell}U_{0}\|_{L^2}+\int^{t}_{0}e^{-c_{2}(t-\tau)}\|\Delta_{q}\Lambda^{\ell}\mathcal{R}(\tau)\|_{L^2}d\tau
 \label{R-E72}
\end{eqnarray}
which leads to
\begin{eqnarray}
\sum_{q\geq0}2^{q(1/2-\ell)}\|\Delta_{q}\Lambda^{\ell}U\|_{L^2}\lesssim \|U_{0}\|_{B^{1/2}_{2,1}}e^{-c_{1}t}+\int^{t}_{0}e^{-c_{2}(t-\tau)}\|\mathcal{R}(\tau)\|_{\dot{B}^{1/2}_{2,1}}d\tau \label{R-E73}
\end{eqnarray}
for $0\leq\ell\leq 1/2$.
Next, we turn to estimate the norm $\|\mathcal{R}(\tau)\|_{\dot{B}^{1/2}_{2,1}}$ as follows
\begin{eqnarray}
\|\mathcal{R}(\tau)\|_{\dot{B}^{1/2}_{2,1}}&=&\|g'(z)z_{x}\|_{\dot{B}^{1/2}_{2,1}}\leq\|z\|_{\dot{B}^{1/2}_{2,1}}\|z\|_{\dot{B}^{3/2}_{2,1}}
\leq\|\Lambda^{\ell}z\|_{\dot{B}^{1/2-\ell}_{2,1}}\|z\|_{B^{3/2}_{2,1}}\nonumber\\ &\lesssim& (1+\tau)^{-\frac{1}{4}-\frac{\ell}{2}}\mathcal{E}_{0}(t)\mathcal{E}_{1}(t), \label{R-E74}
\end{eqnarray}
where Lemma \ref{lem2.1} and Proposition \ref{prop2.2} have been used.

Therefore, together with (\ref{R-E73})-(\ref{R-E74}), we obtain
\begin{eqnarray}
\sum_{q\geq0}2^{q(1/2-\ell)}\|\Delta_{q}\Lambda^{\ell}U\|_{L^2}\lesssim \|U_{0}\|_{B^{1/2}_{2,1}}e^{-c_{1}t}+(1+t)^{-\frac{1}{4}-\frac{\ell}{2}}\mathcal{E}_{0}(t)\mathcal{E}_{1}(t) \label{R-E75}
\end{eqnarray}
for $0\leq\ell\leq 1/2$.\\

\underline{\textit{Step 2: Low-frequency estimate}}

In the following, we proceed with the different low-frequency estimate in comparison with \cite{XK2}, where those analysis remain
only for higher dimensions due to interpolation techniques. Here, the proof involves new observations,
which is divided into two cases.

(i) In the case of $0\leq\ell<1/2$, we have the low-frequency estimate for (\ref{R-E4})-(\ref{R-E5}):
\begin{eqnarray}
\|\Delta_{-1}\Lambda^{\ell}[\mathcal{G}(t)U_{0}]\|_{L^2}\lesssim\|\tilde{w}_{0}\|_{\dot{B}^{-1/2}_{2,\infty}}(1+t)^{-\frac{1}{4}-\frac{\ell}{2}}. \label{R-E76}
\end{eqnarray}
Then it follows from Lemma \ref{lem4.1} that
\begin{eqnarray}
&&\|\Delta_{-1}\Lambda^{\ell}U(t,x)\|_{L^2}\nonumber\\&\leq&\|\Delta_{-1}\Lambda^{\ell}[\mathcal{G}(t)U_{0}]\|_{L^2}
+\int^{t}_{0}\|\Delta_{-1}\Lambda^{\ell}[\mathcal{G}(t-\tau)\mathcal{R}(\tau)]\|_{L^2}d\tau
\nonumber\\&\lesssim&\|U_{0}\|_{\dot{B}^{-1/2}_{2,\infty}}(1+t)^{-\frac{1}{4}-\frac{\ell}{2}}
+I_{1}+I_{2}, \label{R-E77}
\end{eqnarray}
where
\begin{eqnarray*}
I_{1}=\int_{0}^{t/2}\|\Delta_{-1}\Lambda^{\ell}[\mathcal{G}(t-\tau)\mathcal{R}(\tau)]\|_{L^2} d\tau,
\end{eqnarray*}
and
\begin{eqnarray*}
I_{2}=\int_{t/2}^{t}\|\Delta_{-1}\Lambda^{\ell}[\mathcal{G}(t-\tau)\mathcal{R}(\tau)]\|_{L^2} d\tau.
\end{eqnarray*}
For $I_{1}$, noticing that the form $\mathcal{R}(\tau)\approx g(z)_{x}$, we arrive at
\begin{eqnarray}
I_{1}&\lesssim& \int_{0}^{t/2}(1+t-\tau)^{-\frac{1}{4}-\frac{\ell+1}{2}}\|g(z)\|_{\dot{B}^{-1/2}_{2,\infty}}d\tau
 \nonumber\\&\lesssim&
  \int_{0}^{t/2}(1+t-\tau)^{-\frac{1}{4}-\frac{\ell+1}{2}} \|g(z)\|_{L^1}d\tau
 \nonumber\\&\lesssim& \int_{0}^{t/2}(1+t-\tau)^{-\frac{1}{4}-\frac{\ell+1}{2}} \|z\|^2_{L^2}d\tau
  \nonumber\\&\lesssim& \mathcal{E}^2_{1}(t)  \int_{0}^{t/2}(1+t-\tau)^{-\frac{1}{4}-\frac{\ell+1}{2}} (1+\tau)^{-\frac{1}{2}}d\tau
   \nonumber\\&\lesssim&  \mathcal{E}^2_{1}(t)(1+t)^{-\frac{1}{4}-\frac{\ell+1}{2}}\int_{0}^{t/2}(1+\tau)^{-\frac{1}{2}}d\tau
   \nonumber\\&\lesssim& \mathcal{E}^2_{1}(t)(1+t)^{-\frac{1}{4}-\frac{\ell}{2}}, \label{R-E78}
\end{eqnarray}
where we have used the fact that $g(z)=O(z^2)$ and the embeddings $L^{1}\hookrightarrow \dot{B}^{-1/2}_{2,\infty}$ in Lemma \ref{lem2.3} and $B^{1/2}_{2,1}\hookrightarrow L^{2}$.

On the other hand, for $I_{2}$, we have
\begin{eqnarray}
I_{2}&\lesssim& \int_{t/2}^{t}(1+t-\tau)^{-\frac{1}{4}-\frac{\ell}{2}}\|g(z)_{x}\|_{\dot{B}^{-1/2}_{2,\infty}}d\tau. \label{R-E79}
\end{eqnarray}
It follows from Lemma \ref{lem2.1} and Proposition \ref{prop2.4} that
\begin{eqnarray}
\|g(z)_{x}\|_{\dot{B}^{-1/2}_{2,\infty}}\lesssim \|g(z)\|_{\dot{B}^{1/2}_{2,\infty}}\lesssim \|z\|^2_{\dot{B}^{1/2}_{2,\infty}}\lesssim \|z\|^2_{\dot{B}^{1/2}_{2,1}}\lesssim (1+\tau)^{-1}\mathcal{E}^2_{1}(t). \label{R-E80}
\end{eqnarray}
Hence, together with (\ref{R-E79})-(\ref{R-E80}),  we are led to the estimate
\begin{eqnarray}
I_{2}&\lesssim&\mathcal{E}^2_{1}(t)\int_{t/2}^{t}(1+t-\tau)^{-\frac{1}{4}-\frac{\ell}{2}} (1+\tau)^{-1} d\tau \nonumber\\&\lesssim&  \mathcal{E}^2_{1}(t)(1+t)^{-1}\int_{t/2}^{t}(1+t-\tau)^{-\frac{1}{4}-\frac{\ell}{2}}d\tau
\nonumber\\&\lesssim&  \mathcal{E}^2_{1}(t)(1+t)^{-\frac{1}{4}-\frac{\ell}{2}}. \label{R-E81}
\end{eqnarray}
Finally, combing (\ref{R-E77})-(\ref{R-E78}) and (\ref{R-E81}), we conclude that
\begin{eqnarray}
\|\Delta_{-1}\Lambda^{\ell}U(t,x)\|_{L^2}\lesssim \|U_{0}\|_{\dot{B}^{-1/2}_{2,\infty}}(1+t)^{-\frac{1}{4}-\frac{\ell}{2}}+\mathcal{E}^2_{1}(t)(1+t)^{-\frac{1}{4}-\frac{\ell}{2}}. \label{R-E82}
\end{eqnarray}

(ii) In the case of $\ell=1/2$, similar to the procedure leading to (\ref{R-E82}), we can deduce the corresponding nonlinear low-frequency estimate
\begin{eqnarray}
&&\sum_{q<0}\|\dot{\Delta}_{q}\Lambda^{1/2}U\|_{L^2}
\nonumber\\&\leq&\sum_{q<0}
\|\dot{\Delta}_{q}\Lambda^{1/2}[\mathcal{G}(t)U_{0}]\|_{L^2}
+\int^{t}_{0}\sum_{q<0}\|\dot{\Delta}_{q}\Lambda^{1/2}[\mathcal{G}(t-\tau)\mathcal{R}(\tau)]\|_{L^2}d\tau.
\nonumber\\&\lesssim &  \|U_{0}\|_{\dot{B}^{-1/2}_{2,\infty}} (1+t)^{-\frac{1}{2}}+\mathcal{E}^2_{1}(t)(1+t)^{-\frac{1}{2}}. \label{R-E83}
\end{eqnarray}

With these preparations (\ref{R-E75}), (\ref{R-E82})-(\ref{R-E83}) in hand, the desired inequality (\ref{R-E70}) is followed directly by the definitions of $\mathcal{E}_{0}(t)$ and $\mathcal{E}_{1}(t)$.
\end{proof}

\noindent \textbf{\textit{The proof of Proposition \ref{prop4.1}}.}
From Theorem \ref{thm1.1}, we see that
$\mathcal{E}_{0}(t)\lesssim \|U_{0}\|_{B^{3/2}_{2,1}}\lesssim \mathcal{M}_{0}$. Thus, if $\mathcal{M}_{0}$ is sufficient small, it follows from (\ref{R-E70}) that
\begin{eqnarray}
\mathcal{E}_{1}(t)\lesssim \mathcal{M}_{0}+\mathcal{E}^{2}_{1}(t), \label{R-E84}
\end{eqnarray}
which implies that $\mathcal{E}(t)\lesssim \mathcal{M}_{0}$ by the standard method, provided that $\mathcal{M}_{0}$ is sufficient small. Consequently, we obtain the decay estimate in Proposition \ref{prop4.1}. $\square$

\section{Appendix}\setcounter{equation}{0}\label{sec:5}
For convenience of reader, in this section, we review the Littlewood--Paley
decomposition and definitions for Besov spaces and Chemin-Lerner spaces in $\mathbb{R}^{n}(n\geq1)$, see
\cite{BCD} for more details.

Let ($\varphi, \chi)$ is a
couple of smooth functions valued in [0,1] such that $\varphi$ is
supported in the shell
$\textbf{C}(0,\frac{3}{4},\frac{8}{3})=\{\xi\in\mathbb{R}^{n}|\frac{3}{4}\leq|\xi|\leq\frac{8}{3}\}$,
$\chi$ is supported in the ball $\textbf{B}(0,\frac{4}{3})=
\{\xi\in\mathbb{R}^{n}||\xi|\leq\frac{4}{3}\}$ satisfying
$$
\chi(\xi)+\sum_{q\in\mathbb{N}}\varphi(2^{-q}\xi)=1,\ \ \ \ q\in
\mathbb{N},\ \ \xi\in\mathbb{R}^{n}
$$
and
$$
\sum_{k\in\mathbb{Z}}\varphi(2^{-k}\xi)=1,\ \ \ \ k\in \mathbb{Z},\
\ \xi\in\mathbb{R}^{n}\setminus\{0\}.
$$
For $f\in\mathcal{S'}$(the set of temperate distributions
which is the dual of the Schwarz class $\mathcal{S}$),  define
$$
\Delta_{-1}f:=\chi(D)f=\mathcal{F}^{-1}(\chi(\xi)\mathcal{F}f),\
\Delta_{q}f:=0 \ \  \mbox{for}\ \  q\leq-2;
$$
$$
\Delta_{q}f:=\varphi(2^{-q}D)f=\mathcal{F}^{-1}(\varphi(2^{-q}|\xi|)\mathcal{F}f)\
\  \mbox{for}\ \  q\geq0;
$$
$$
\dot{\Delta}_{q}f:=\varphi(2^{-q}D)f=\mathcal{F}^{-1}(\varphi(2^{-q}|\xi|)\mathcal{F}f)\
\  \mbox{for}\ \  q\in\mathbb{Z},
$$
where $\mathcal{F}f$, $\mathcal{F}^{-1}f$ represent the Fourier
transform and the inverse Fourier transform on $f$, respectively. Observe that
the operator $\dot{\Delta}_{q}$ coincides with $\Delta_{q}$ for $q\geq0$.

Denote by $\mathcal{S}'_{0}:=\mathcal{S}'/\mathcal{P}$ the tempered
distributions modulo polynomials $\mathcal{P}$. We first give the definition of homogeneous Besov spaces.
\begin{defn}\label{defn4.1}
For $s\in \mathbb{R}$ and $1\leq p,r\leq\infty,$ the homogeneous
Besov spaces $\dot{B}^{s}_{p,r}$ is defined by
$$\dot{B}^{s}_{p,r}=\{f\in S'_{0}:\|f\|_{\dot{B}^{s}_{p,r}}<\infty\},$$
where
$$\|f\|_{\dot{B}^{s}_{p,r}}
=\begin{cases}\Big(\sum_{q\in\mathbb{Z}}(2^{qs}\|\dot{\Delta}_{q}f\|_{L^p})^{r}\Big)^{1/r},\
\ r<\infty, \\ \sup_{q\in\mathbb{Z}}
2^{qs}\|\dot{\Delta}_{q}f\|_{L^p},\ \ r=\infty.\end{cases}
$$\end{defn}

Similarly, the definition of inhomogeneous Besov spaces is stated as follows.
\begin{defn}\label{defn4.2}
For $s\in \mathbb{R}$ and $1\leq p,r\leq\infty,$ the inhomogeneous
Besov spaces $B^{s}_{p,r}$ is defined by
$$B^{s}_{p,r}=\{f\in S':\|f\|_{B^{s}_{p,r}}<\infty\},$$
where
$$\|f\|_{B^{s}_{p,r}}
=\begin{cases}\Big(\sum_{q=-1}^{\infty}(2^{qs}\|\Delta_{q}f\|_{L^p})^{r}\Big)^{1/r},\
\ r<\infty, \\ \sup_{q\geq-1} 2^{qs}\|\Delta_{q}f\|_{L^p},\ \
r=\infty.\end{cases}$$
\end{defn}

On the other hand, we also present the definition of Chemin-Lerner
spaces first initialled by J.-Y. Chemin and N. Lerner
\cite{CL}, which are the refinement of the space-time mixed spaces
$L^{\theta}_{T}(\dot{B}^{s}_{p,r})$ or
$L^{\theta}_{T}(B^{s}_{p,r})$.

\begin{defn}\label{defn4.3}
For $T>0, s\in\mathbb{R}, 1\leq r,\theta\leq\infty$, the homogeneous
mixed Chemin-Lerner spaces
$\widetilde{L}^{\theta}_{T}(\dot{B}^{s}_{p,r})$ is defined by
$$\widetilde{L}^{\theta}_{T}(\dot{B}^{s}_{p,r}):
=\{f\in
L^{\theta}(0,T;\mathcal{S}'_{0}):\|f\|_{\widetilde{L}^{\theta}_{T}(\dot{B}^{s}_{p,r})}<+\infty\},$$
where
$$\|f\|_{\widetilde{L}^{\theta}_{T}(\dot{B}^{s}_{p,r})}:=\Big(\sum_{q\in\mathbb{Z}}(2^{qs}\|\dot{\Delta}_{q}f\|_{L^{\theta}_{T}(L^{p})})^{r}\Big)^{\frac{1}{r}}$$
with the usual convention if $r=\infty$.
\end{defn}

\begin{defn}\label{defn4.4}
For $T>0, s\in\mathbb{R}, 1\leq r,\theta\leq\infty$, the
inhomogeneous Chemin-Lerner spaces
$\widetilde{L}^{\theta}_{T}(B^{s}_{p,r})$ is defined by
$$\widetilde{L}^{\theta}_{T}(B^{s}_{p,r}):
=\{f\in
L^{\theta}(0,T;\mathcal{S}'):\|f\|_{\widetilde{L}^{\theta}_{T}(B^{s}_{p,r})}<+\infty\},$$
where
$$\|f\|_{\widetilde{L}^{\theta}_{T}(B^{s}_{p,r})}:=\Big(\sum_{q\geq-1}(2^{qs}\|\Delta_{q}f\|_{L^{\theta}_{T}(L^{p})})^{r}\Big)^{\frac{1}{r}}$$
with the usual convention if $r=\infty$.
\end{defn}

We further define
$$\widetilde{\mathcal{C}}_{T}(B^{s}_{p,r}):=\widetilde{L}^{\infty}_{T}(B^{s}_{p,r})\cap\mathcal{C}([0,T],B^{s}_{p,r})
$$ and $$\widetilde{\mathcal{C}}^1_{T}(B^{s}_{p,r}):=\{f\in\mathcal{C}^1([0,T],B^{s}_{p,r})|\partial_{t}f\in\widetilde{L}^{\infty}_{T}(B^{s}_{p,r})\},$$
where the index $T$ will be omitted when $T=+\infty$.

By Minkowski's inequality, $\widetilde{L}^{\theta}_{T}(B^{s}_{p,r})$ may
be linked with the usual space-time mixed spaces $L^{\theta}_{T}(B^{s}_{p,r})$.
\begin{rem}\label{rem5.1}
It holds that
$$\|f\|_{\widetilde{L}^{\theta}_{T}(B^{s}_{p,r})}\leq\|f\|_{L^{\theta}_{T}(B^{s}_{p,r})}\,\,\,
\mbox{if}\,\, r\geq\theta;\ \ \ \
\|f\|_{\widetilde{L}^{\theta}_{T}(B^{s}_{p,r})}\geq\|f\|_{L^{\theta}_{T}(B^{s}_{p,r})}\,\,\,
\mbox{if}\,\, r\leq\theta.
$$\end{rem}

\section*{Acknowledgments}
J. Xu is partially supported by the Program for New Century Excellent
Talents in University (NCET-13-0857), Special Foundation of China Postdoctoral
Science Foundation (2012T50466) and the NUAA Fundamental
Research Funds (NS2013076). The work is also partially supported by
Grant-in-Aid for Scientific Researches (S) 25220702 and (A) 22244009.

\end{document}